\documentclass[10pt]{article}

\usepackage{lmodern}
\usepackage[T1]{fontenc}
\usepackage{amsmath}
\usepackage{amsthm}
\usepackage{amssymb}
\usepackage{mathabx} 
\usepackage{url}
\usepackage{latexsym}
\usepackage{titlefoot}
\usepackage[small]{titlesec}
\usepackage{units} 
\usepackage[small,it]{caption}
\usepackage{xspace}
\usepackage[mathscr]{euscript} 

\usepackage[symbol]{footmisc}

\usepackage[square,comma,numbers,sort&compress]{natbib}

\usepackage{graphicx} 

\usepackage{tikz}

\usepackage{color}
\definecolor{lightgray}{rgb}{0.8, 0.8, 0.8}
\definecolor{darkgray}{rgb}{0.6, 0.6, 0.6}

\usepackage[bookmarks]{hyperref}
\hypersetup{
	colorlinks=true,
	linkcolor=black,
	anchorcolor=black,
	citecolor=black,
	urlcolor=black,
	pdfpagemode=UseThumbs,
	pdftitle={How many pop-stacks does it take to sort of permutation?},
	pdfsubject={Combinatorics},
	pdfauthor={Michael Albert and Vincent Vatter},
}

\newcounter{todocounter}

\newcommand{\minisec}[1]{\noindent{\sc #1.}}

\theoremstyle{plain}
\newtheorem{theorem}{Theorem}
\newtheorem{proposition}[theorem]{Proposition}

\makeatletter
\newfont{\footsc}{cmcsc10 at 8truept}
\newfont{\footbf}{cmbx10 at 8truept}
\newfont{\footrm}{cmr10 at 10truept}
\renewenvironment{abstract}{
	\begin{list}{}%
	{\setlength{\rightmargin}{1in}%
	\setlength{\leftmargin}{1in}}%
	\item[]\ignorespaces\begin{small}}%
	{\end{small}\unskip\end{list}%
}

\datefoot{\today}
\amssubj{68P10 (primary), 05A05, 68R05 (secondary)}

\newpagestyle{main}[\small]{
	\headrule
	\sethead[\usepage][][]
	{\sc How Many Pop-Stacks Does it Take to Sort a Permutation?}{}{\usepage}
}

\title{\sc How Many Pop-Stacks Does it Take to Sort a Permutation?}

\author{\centering
\begin{tabular}{ccc}
Michael Albert
&\rule{0pt}{0pt}&
Vincent Vatter%
\footnote{Vatter's research was partially supported by the Simons Foundation via award number 636113.}\\[-0.25ex]
\small Department of Computer Science
&&
\small Department of Mathematics\\[-0.5ex]
\small University of Otago
&&
\small University of Florida\\[-0.5ex]
\small Dunedin, New Zealand
&&
\small Gainesville, Florida USA\\[-1.5ex]
\end{tabular}
}

\titleformat{\section}{\large\sc}{\thesection.}{1em}{}
\date{}
 
\begin{document}
\maketitle

\pagestyle{main}

\begin{abstract}
{\sc Abstract.}
Pop-stacks are variants of stacks that were introduced by Avis and Newborn in 1981. Coincidentally, a 1982 result of Unger implies that every permutation of length $n$ can be sorted by $n-1$ passes through a deterministic pop-stack. We give a new proof of this result inspired by Knuth's zero-one principle.
\end{abstract}

\section{Introduction}

We are concerned here with the sorting of permutations. We take a \emph{permutation of length $n$} to be a word over the alphabet $\{1,2,\dots,n\}$ in which every symbol occurs precisely once, and \emph{sorting} the permutation $\pi=\pi(1)\cdots\pi(n)$ of length $n$ means transforming it into the \emph{identity} $12\cdots n$.

The mathematical study of stack sorting was initiated by Knuth in the first volume of \emph{The Art of Computer Programming}~\cite[Section~2.2.1]{knuth:the-art-of-comp:1}, where he considered stacks, queues, and deques. In particular, a \emph{stack} is a first-in last-out (filo) sorting device with the operations \emph{push} (add the next entry of the input to the top of the stack) and \emph{pop} (remove the top entry of the stack and place it at the end of the output). The stack sorting of a permutation is depicted in Figure~\ref{fig-stack-sorting}. 

\newcommand{\stackdraw}{%
	\draw [very thick, line cap=round, darkgray]
		(6,5)--(1,5)--(1,0)--(-1,0)--(-1,5)--(-6,5);
	\node at (-3.5,5) [below] {output};
	\node at (3.5,5) [below] {input};
}
\newcommand{\stackfill}[5]{%
	\node at (-6,6) [anchor=west] {$#1$};
	\node at (0,-0.1) [anchor=south] {$#2$};
	\node at (0,1.1) [anchor=south] {$#3$};
	\node at (0,2.3) [anchor=south] {$#4$};
	\node at (6,6) [anchor=east] {$#5$};
}
\newcommand{\stackpush}{%
	\draw [rounded corners=5, ->] (2,6)--({1/3},6)--({1/3},4);
}
\newcommand{\stackpop}{%
	\draw [rounded corners=5, ->] ({-1/3},4)--({-1/3},6)--(-2,6);
}

\begin{figure}[t]
\begin{center}
\begin{footnotesize}
\begin{tabular}{c@{\hskip 0.5in}c@{\hskip 0.5in}c@{\hskip 0.5in}c}
\begin{tikzpicture}[scale=0.2, baseline=(current bounding box.center)]
	\stackdraw
	\stackfill{}{3}{}{}{124}
\end{tikzpicture}
&
\begin{tikzpicture}[scale=0.2, baseline=(current bounding box.center)]
	\stackdraw
	\stackfill{}{3}{1}{}{24}
	\stackpop
\end{tikzpicture}
&
\begin{tikzpicture}[scale=0.2, baseline=(current bounding box.center)]
	\stackdraw
	\stackfill{1}{3}{}{}{24}
	\stackpush
\end{tikzpicture}
&
\begin{tikzpicture}[scale=0.2, baseline=(current bounding box.center)]
	\stackdraw
	\stackfill{1}{3}{2}{}{4}
	\stackpop
\end{tikzpicture}
\end{tabular}
\end{footnotesize}
\end{center}
\caption{In the process of sorting the permutation $3124$ with a stack, we initially push the $3$ onto the stack, then perform a push, pop, and push as shown above, and then finish the sorting by popping the $2$ and the $3$ and then pushing and popping the $4$.}
\label{fig-stack-sorting}
\end{figure}
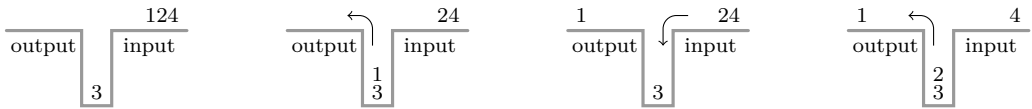

In the case of sorting with a single stack, there is a \emph{canonical stack sorting algorithm}. This algorithm can sort every permutation that can be sorted with a single stack, and is governed by the following two deterministic rules.
\begin{enumerate}
	\item[($1$)] If the next entry in the input is less than the topmost entry of the stack (or the stack is empty), push this entry onto the top of the stack.
	\item[($2$)] Otherwise, pop the topmost entry of the stack.
\end{enumerate}
In fact, this algorithm can be applied to any word of distinct symbols to produce another word on the same set of symbols. We denote this operation by $S$, so $S(\pi)$ is the output obtained after applying the canonical stack sorting algorithm to the permutation $\pi$. It follows readily that if $\pi=\lambda n \rho$, where where $n$ is the largest symbol of $\pi$, $\lambda$ is the word consisting  of the entries of $\pi$ to the left of $n$, and $\rho$ is the word consisting of the entries to the right of $n$, then
\begin{equation}\tag{\textsection}\label{eqn-stack-operator}
	S(\pi)
	=
	S(\lambda n \rho)
	=
	S(\lambda) S(\rho) n.
\end{equation}	
For example,
\[
	S(463152)
	=
	S(4)\, S(3152)\, 6
	=
	4\, S(31)\, S(2)\, 56
	=
	413256.
\]

From our comments above, the permutation $\pi$ is \emph{stack-sortable} if and only if $S(\pi)$ is the identity. Knuth showed that these are precisely the permutations that avoid the pattern $231$. If $S^t(\pi)$ is the identity, then we say that $\pi$ can be \emph{sorted by $t$ passes through a stack} or that it is \emph{West-$t$-stack-sortable}. This name is a reference to West~\cite{west:sorting-twice-t:}, who first studied these permutations and conjectured that the number of West-$2$-stack-sortable permutations of length $n$ has the formula $2(3n)!/((n+1)!(2n+1)!)$, a conjecture first verified by Zeilberger~\cite{zeilberger:a-proof-of-juli:}.

We are concerned here with a restricted version of a stack that was first introduced by Avis and Newborn~\cite{avis:on-pop-stacks-i:} in 1981. A \emph{pop-stack} is identical to a stack except that when a pop is performed, the entire contents of the pop-stack must be popped. For example, one can check that the permutation $312$ is stack-sortable but not pop-stack-sortable. A simple adaptation of the algorithm for stack sorting yields the \emph{canonical pop-stack sorting algorithm}, which is governed by the following two deterministic rules.
\begin{enumerate}
	\item[($1$)] If the next entry in the input is less than the topmost entry of the pop-stack (or the pop-stack is empty), push this entry onto the top of the pop-stack.
	\item[($2'$)] Otherwise, pop all of the entries of the pop-stack.
\end{enumerate}
We denote the operation of performing this algorithm on the permutation $\pi$ by $P(\pi)$. It can be checked that this canonical pop-stack sorting algorithm can sort every permutation that can be sorted by a pop-stack, so the permutation $\pi$ is pop-stack-sortable if and only if $P(\pi)$ is the identity. It can also be seen that the pop-stack-sortable permutations are precisely those avoiding the patterns $231$ and $312$ (the \emph{layered} permutations).

Analogous to \eqref{eqn-stack-operator}, if the entries of $\pi$ can be partitioned into $\ell$ maximal decreasing runs (contiguous subsequences whose entries are decreasing) $\pi_1\pi_2\cdots\pi_\ell$, then
\begin{equation}\tag{\textparagraph}\label{eqn-pop-stack-operator}
	P(\pi)=\pi_1^{\textrm{r}}\pi_2^{\textrm{r}}\cdots\pi_\ell^{\textrm{r}},
\end{equation}	
where $\pi_i^{\textrm{r}}$ is the result of reversing the word $\pi_i$. For example,
\[
	P(463152)
	=
	(4)^{\textrm{r}}(631)^{\textrm{r}}(52)^{\textrm{r}}
	=
	413625.
\]

We are interested in how many passes through a pop-stack (using the canonical pop-stack sorting algorithm) are necessary in order to ensure that every permutation of length $n$ can be sorted. It is obvious from \eqref{eqn-stack-operator} that the analogous answer for a stack is $n-1$, because the symbol $n$ is in its correct position (last) in $S(\pi)$, both symbols $n-1$ and $n$ are in their correct positions in $S^2(\pi)$, and so on. It is not a priori obvious that the same bound suffices for sorting with pop-stacks, but this is indeed the case.

\begin{theorem}
\label{thm-n-1-pop-stacks-suffice}
For every permutation $\pi$ of length $n$, $P^{n-1}(\pi)$ is equal to the identity.
\end{theorem}

The bound in Theorem~\ref{thm-n-1-pop-stacks-suffice} can be seen to be best possible by considering the permutation $23\cdots n1$. Indeed, there are a multitude of permutations requiring the maximum number of passes through a pop-stack to be sorted. These permutations do not appear to have been counted yet.

In a paper published only a year after Avis and Newborn's introduction of pop-stacks, but with completely different motivations, Ungar~\cite{ungar:2n-noncollinear:} proved a result that implies Theorem~\ref{thm-n-1-pop-stacks-suffice}, via \eqref{eqn-pop-stack-operator}. Specifically, he proved that no matter what permutation of length $n$ one starts with, one can reach the permutation $n(n-1)\cdots 21$ in at most $n-1$ ``moves'', where a move consists of reversing every maximal increasing run of the permutation. Our motivation for this note is to prove Theorem~\ref{thm-n-1-pop-stacks-suffice} via a series of relaxations from pop-stack-sorting to a nondeterministic sort we call tumble, and then to a deterministic sort we call flip. Fundamentally, our proof and that presented by Ungar are quite similar. However, we believe this change of perspective makes it easier to appreciate the mechanics of the proof, and that the intermediate steps may be of further use.

\section{From Pop-Stacks to Tumbles}


Our proof of Theorem~\ref{thm-n-1-pop-stacks-suffice} is inspired by Knuth's zero-one principle~\cite[Theorem~Z of Section~5.3.4]{knuth:the-art-of-comp:3}, which states that if a certain type of sorting network is able to sort all binary words, then it is able to sort all words over any ordered alphabet. While this result does not apply in our context, we do utilize a part of its proof. Given a permutation $\pi$ of length $n$ and an integer $0\le k\le n$, we define the word $\pi\vert_k$ by
\[
	\pi\vert_k(i)
	=
	\left\{
	\begin{array}{ll}
		0&\text{if $\pi(i) < k$, or}\\
		1&\text{if $\pi(i) \ge k$.}
	\end{array}
	\right.
\]
for all $1\le i\le n$. For example,
\[
	463152\vert_3
	=
	110010.
\]

We also introduce a nondeterministic sorting operation $T$ on binary words, which we call \emph{tumble}. Given a binary word $w$, the tumble operator performs the following two steps.
\begin{enumerate}
	\item[(1)] Choose a collection of non-overlapping factors (contiguous subsequences) of the form $1^+0^+$ (that is, a string of one or more ones followed by a string of one or more zeros) which together contain every $10$ factor in $w$.
	\item[(2)] Reverse each of those factors in place.
\end{enumerate}
Since tumble is nondeterministic, it produces multiple outputs from a single input. We denote by $T(w)$ the \emph{set} of all possible outputs of the tumble operator given the input $w$. For example, there are four possible outputs when the tumble operator is applied to the word $110010$, depending on the four different ways to choose the first factor of the form $1^+0^+$, and thus we have
\begin{eqnarray*}
	T(110010)
	&=&
	\{(1100)^{\textrm{r}}(10)^{\textrm{r}},
	  (110)^{\textrm{r}}0(10)^{\textrm{r}},
	  1(100)^{\textrm{r}}(10)^{\textrm{r}},
	  1(10)^{\textrm{r}}(10)^{\textrm{r}}\}\\
	&=&
	\{001101,
	  011001,
	  100101,
	  101001\}.
\end{eqnarray*}

Our result below establishes the connection between pop-stacks and tumbles. The result holds for the example above because
\[
	P(463152)\vert_3
	=
	413625\vert_3
	=
	100101
	\in
	T(110010)
	=
	T(463152\vert_3).
\]

\begin{proposition}
\label{prop-tumbles}
For every permutation $\pi$ of length $n$ and every integer $0\le k\le n$,
\[
	P(\pi)\vert_k \in T(\pi\vert_k).
\]
\end{proposition}
\begin{proof}
Given any partition of the entries of $\pi$ into factors as $\pi=\pi_1\pi_2\cdots\pi_\ell$, we have
\[
	\pi\vert_k
	=
	(\pi_1\pi_2\cdots\pi_\ell)\vert_k
	=
	\left(\pi_1\vert_k\right) \left(\pi_2\vert_k\right) \cdots \left(\pi_\ell\vert_k\right).
\]
If we take the $\pi_i$ in the above expression to be the maximal decreasing runs of $\pi$, then every $10$ factor of $\pi\vert_k$ must occur within one of the factors $\pi_i\vert_k$, and each of the factors $\pi_i\vert_k$ must be of the form $1^+$, $1^+0^+$, or $0^+$, depending on how its first and last entries compare to $k$. Therefore the tumble operator is allowed to choose precisely the $\pi_i\vert_k$ of the form $1^+0^+$ as its collection of non-overlapping factors and reverse each of those factors in place. From here we need only note that
\[
	P(\pi)\vert_k
	=
	(\pi_1^{\textrm{r}}\pi_2^{\textrm{r}}\cdots\pi_\ell^{\textrm{r}})\vert_k
	=
	\left(\pi_1^{\textrm{r}}\vert_k\right) \left(\pi_2^{\textrm{r}}\vert_k\right) \cdots \left(\pi_\ell^{\textrm{r}}\vert_k\right)
	\in
	T(\pi\vert_k)
\]
to complete the proof.
\end{proof}

Proposition~\ref{prop-tumbles} shows that for all $k$ and $t$,
\[
	P^t(\pi)\vert_k
	=
	P(P^{t-1}(\pi))\vert_k
	\in
	T(P^{t-1}(\pi)\vert_k),
\]
and thus it follows by induction on $t$ that $P^t(\pi)\vert_k\in T^t(\pi\vert_k)$.

\section{From Tumbles to Flips}

Given a binary word $w$, the \emph{flip} sorting operator simply reverses every $10$ factor. Thus flip is deterministic, and we denote its output by $F(w)$. It should be obvious from the definitions of flip and tumble that for every binary word $w$, $F(w)\in T(w)$, or in other words, a flip is always a valid tumble.

We also consider a slightly nonstandard partial order on binary words. Given binary words $u$ and $v$, we write $u\trianglelefteq v$ if, for every $i$, the $i$th zero of $u$ (if it exists) occurs in the same position or before the $i$th zero of $v$. For example, in this order $001101\trianglelefteq 011001$ because the zeros of $001101$ occur in the indices $1$, $2$, and $5$ while the zeros of $011001$ occur in the indices $1$, $4$, and $5$. However, the words $u=011100$ and $v=101010$ are incomparable under this order because the first zero of $u$ occurs before the first zero of $v$, but the second zero of $u$ occurs after the second zero of $v$. We use this order only on words $u$ and $v$ having the same number of zeros, though strictly speaking the definition requires only that $v$ should have at least as many zeros as $u$.

\begin{proposition}
\label{prop-flips}
If $u\trianglelefteq v$, then $F(u)\trianglelefteq F(v)$.
\end{proposition}
\begin{proof}
Suppose, to the contrary, that $u\trianglelefteq v$ but that $F(u)\not\trianglelefteq F(v)$. As the flip operator moves each zero at most one position to the left, this means that there must be some index $i$ such that the $i$th zeros of $u$ and $v$ occur in the same position, but the flip operator moves the $i$th zero of $v$ one position to the left while it does not move the $i$th zero of $u$. For that to be the case, it must be that the $i$th zero of $v$ is immediately preceded by a one, while the $i$th zero of $u$ is not (and thus is immediately preceded by a zero). However, this implies that the $i-1$st zero of $v$ occurs before the $i-1$st zero of $u$, and that contradicts our hypothesis that $u\trianglelefteq v$.
\end{proof}

We noted above that flip is a valid tumble, but it is actually the worst tumble in the sense of this partial order. Indeed, in every application of tumble, every zero is moved at least as far to the left as it is moved by flip (identifying zeros by their position in left to right order). Thus we have $u\trianglelefteq F(w)$ for every $u\in T(w)$, and we indicate this by writing $T(w)\trianglelefteq F(w)$. It follows by induction that $T^t(w)\trianglelefteq F^t(w)$ for all positive integers $t$.

Since pop-stacks perform a sort of tumble (modulo the compression to binary words), and flips are the worst tumbles, we proceed to consider how long it takes the flip operator to sort a binary word. Given any binary word $w$ with $a$ zeros and $b$ ones, we have $w\trianglelefteq 1^b0^a$. Proposition~\ref{prop-flips} and induction therefore imply that $w$ requires at most as many flips to sort (into $0^a1^b$) as $1^b0^a$ requires, and this number is given by the following result.

\begin{proposition}
\label{prop-flips-number}
The word $1^b0^a$ requires precisely $a+b-1$ flips to sort.
\end{proposition}
\begin{proof}
The rightmost zero of $1^b0^a$ is not moved for the first $a-1$ iterations of the flip operator. Thereafter, it is moved one position left by every iteration. As this entry needs to move $b$ positions to the left, it arrives to its final location after precisely $a+b-1$ iterations of the flip operator. Once the rightmost zero arrives in its final location, then all the other zeros (which occur before it by definition) must be in their final locations, and so the word has been sorted.
\end{proof}

In particular, Proposition~\ref{prop-flips-number} shows that every binary word of length $n$ can be sorted by $n-1$ applications of the flip operator. Combining this with observations we have already made, it follows that for all permutations $\pi$ of length $n$ and integers $0\le k\le n$,
\[
	P^{n-1}(\pi)\vert_k
	\in
	T^{n-1}(\pi\vert_k)
	\trianglelefteq
	F^{n-1}(\pi\vert_k)
	=
	0^k1^{n-k}.
\]
This implies that, for every choice of $k$, all entries of $P^{n-1}(\pi)$ of value less than $k$ occur before all entries of $P^{n-1}(\pi)$ of value at least $k$. The only way that can occur is if $P^{n-1}(\pi)$ is in fact the identity, proving Theorem~\ref{thm-n-1-pop-stacks-suffice}.

\section{Concluding Remarks}

We conclude with a discussion of the variations on sorting with stacks in series, as there has been some confusion about this in the literature. There are two distinct notions of what it means for the permutation $\pi$ to be ``sorted by stacks in series''. What we defined to be West-$t$-stack-sortability in the introduction means that $S^t(\pi)$ is the identity. The other notion, which we believe deserves the name \emph{$t$-stack-sortability}, states merely that $t$ stacks connected in series can sort $\pi$, if no restrictions are placed on the algorithm to be used (that is, if the stacks are allowed to function nondeterministically).

Enumerations and characterizations are much easier to obtain for West-$t$-stack-sortability than for the more powerful $t$-stack-sortability. For an example of the stark difference between the two definitions, note that it is trivial to determine if $\pi$ is West-$2$-stack-sortable%
\footnote{Characterizing the West-$t$-stack-sortable permutations in terms of pattern avoidance is more difficult; this was done for $t=2$ by West~\cite{west:sorting-twice-t:} and for $t=3$ by \'Ulfarsson~\cite{ulfarsson:describing-west:}.}%
---one simply computes $S^2(\pi)$---but it was only recently proved by Pierrot and Rossin~\cite{pierrot:2-stack-sorting:} that it can be determined in polynomial time if $\pi$ is $2$-stack-sortable.

\begin{figure}
	\begin{footnotesize}
	\begin{center}
	\begin{tikzpicture}[scale=0.2, baseline=(current bounding box.center)]
		\draw [very thick, line cap=round, darkgray]
			(12,5)--(7,5)--(7,0)--(5,0)--(5,5);
		\draw [very thick, line cap=round, darkgray]
			(-2.5,5)--(2.5,5);
		\draw [very thick, line cap=round, darkgray]
			(-12,5)--(-7,5)--(-7,0)--(-5,0)--(-5,5);
		\node at (-9.5,5) [below] {output};
		\node at (9.5,5) [below] {input};
		\node at (0,5) [below] {queue};
		\draw [rounded corners=5, ->] (9,6)--({7-2/3},6)--({7-2/3},3);
		\draw [rounded corners=5, ->] ({5+2/3},3)--({5+2/3},6)--(2.5,6);
		\draw [rounded corners=5, ->] (-2.5,6)--({-5-2/3},6)--({-5-2/3},3);
		\draw [rounded corners=5, ->] ({-7+2/3},3)--({-7+2/3},6)--(-9,6);
	\end{tikzpicture}
\quad\quad\quad\quad
	\begin{tikzpicture}[scale=0.2, baseline=(current bounding box.center)]
		\draw [very thick, line cap=round, darkgray]
			(9,5)--(4,5)--(4,0)--(2,0)--(2,5)--(0,5)
			--(-2,5)--(-2,0)--(-4,0)--(-4,5)--(-9,5);
		\node at (-6.5,5) [below] {output};
		\node at (6.5,5) [below] {input};
		\draw [rounded corners=5, ->] (6,6)--({4-2/3},6)--({4-2/3},3);
		\draw [rounded corners=5, ->] ({2+2/3},3)--({2+2/3},6)--({-2-2/3},6)--({-2-2/3},3);
		\draw [rounded corners=5, ->] ({-4+2/3},3)--({-4+2/3},6)--(-6,6);
	\end{tikzpicture}
	\end{center}
	\end{footnotesize}
\caption{Two ways to connect stacks and pop-stacks in series: with and without a queue in between.}
\label{fig-2-stacks-series-queue}
\end{figure}
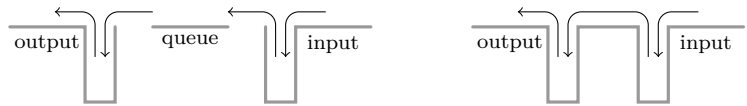

An alternative viewpoint on $t$-stack-sortability arises if we define the $\mathscr{S}(\pi)$ to be the \emph{set} of all permutations that $\pi$ can be transformed into by a stack operating nondeterministically, i.e., not necessarily following the canonical stack sorting algorithm. For example, while a single stack cannot sort $231$, it can transform it into all other permutations of length $3$, so we have
\[
	\mathscr{S}(231)=\{132, 213, 231, 312, 321\}.
\]
Then $\mathscr{S}^t(\pi)$ is the set of all permutations that $\pi$ can be transformed into after $t$ passes through a nondeterministic stack, and $\pi$ can be sorted in this way if the identity lies in $\mathscr{S}^t(\pi)$. Technically, this corresponds to a machine consisting of $t$ stacks in series with queues in between, as shown on the left of Figure~\ref{fig-2-stacks-series-queue}, but it is not hard to see that this machine can sort the same permutations as the machine consisting of $t$ stacks connected in series without queues in between, as shown on the right of Figure~\ref{fig-2-stacks-series-queue}. Therefore this definition of $t$-stack-sortability agrees with our previous definition.

In the case of pop-stack sorting, the situation is more complicated. We suggest that if $P^t(\pi)$ is the identity, then $\pi$ is \emph{West-$t$-pop-stack-sortable} (although West has never studied pop-stacks). Thus Theorem~\ref{thm-n-1-pop-stacks-suffice} states that every permutation of length $n$ is West-$(n-1)$-pop-stack-sortable. There has been some recent research on this notion of sorting with pop-stacks: Pudwell and Smith~\cite{pudwell:two-stack-sorti:} characterize and enumerate the West-$2$-pop-stack-sortable permutations, while Claesson and Gu\dh{}mundsson~\cite{claesson:enumerating-per:} prove that the set of West-$t$-pop-stack-sortable permutations has a rational generating function for all $t$ and give an algorithm for its computation.

There are two different interpretations of connecting pop-stacks (acting nondeterministally) in series, because it matters if there are queues in between. In other words, sorting with two pop-stacks connected without a queue in between (as on the right of Figure~\ref{fig-2-stacks-series-queue}) is strictly less powerful than sorting with one pop-stack and then passing that output into another pop-stack (that is, sorting with two pop-stacks connected with a queue in between, as on the left of Figure~\ref{fig-2-stacks-series-queue}).

Avis and Newborn~\cite{avis:on-pop-stacks-i:} themselves considered pop-stacks in series \emph{without} queues in between them, and thus one of their results states that only certain permutations (which are now called the separable permutations) can be sorted by \emph{any} number of pop-stacks in series. Indeed, in this interpretation, when the first pop-stack performs a pop, all of its contents move immediately into the second pop-stack. Thus no matter how many pop-stacks one has access to, if they are connected without queues between them, once two entries lie in one pop-stack together they will cohabitate all later pop-stacks. We refer to Atkinson and Stitt~\cite[Section 6.3]{atkinson:restricted-perm:wreath} for a redevelopment of the Avis--Newborn results with the more recent machinery of the substitution decomposition.

Connecting pop-stacks in series \emph{with} queues between them has received little attention, although Atkinson and Stitt~\cite[Section 6.4]{atkinson:restricted-perm:wreath} characterize the permutations that can be sorted by $2$ pop-stacks ``in genuine series'' as they call it, and find their (rational) generating function. This version of the problem can be restated by introducing a nondeterministic pop-stack sorting operator as we did above for stacks. We define $\mathscr{P}(\pi)$ to be the set of all possible outputs if $\pi$ is given as the input to a pop-stack operating nondeterministically. For example, one can check that
\[
	\mathscr{P}(231)=\{132, 213, 231, 321\}
\]
because in addition to not being able to sort $231$, a pop-stack also cannot transform it into $312$. With this definition, $t$ pop-stacks connected in series with queues between them can sort $\pi$ if and only if the identity lies in $\mathscr{P}^t(\pi)$.

On a different but related note, in their investigation of the pop-stack operator, Asinowski, Banderier, Billey, Hackl, and Linusson~\cite{asinowski:pop-stack-sorti:} define the permutation $\pi$ to be \emph{pop-stacked} if $\pi=P(\sigma)$ for some $\sigma$. Further research in this direction has performed by Claesson, Gu\dh{}mundsson, and Pantone~\cite{claesson:counting-pop-st:} and Asinowski, Banderier, and Hackl~\cite{asinowski:flip-sort-and-c:}.

\bigskip

\minisec{Acknowledgements}
We thank Mikl\'os B\'ona for interesting us in this problem. 


\def\cprime{$'$}

\end{document}